\documentclass[11pt, reqno]{amsart}
\usepackage{amsfonts}
\usepackage{amsmath}
\usepackage{amssymb, color}
\usepackage{amscd}
\usepackage[mathscr]{eucal}
\usepackage{url}

\usepackage{mathtools}

\allowdisplaybreaks[1]

\renewcommand*{\bar}{\overline}
\newcommand{\gfp}[1]{\Gamma_p{\left({#1}\right)}}
\newcommand{\biggfp}[1]{\Gamma_p{\bigl({#1}\bigr)}}

\theoremstyle{plain}

\newtheorem{theorem}{Theorem}[section]

\newtheorem{prop}[theorem]{Proposition}
\newtheorem{cor}[theorem]{Corollary}
\theoremstyle{definition}
\newtheorem{defi}[theorem]{Definition}

\numberwithin{equation}{section}

\makeatletter
\def\imod#1{\allowbreak\mkern5mu({\operator@font mod}\,\,#1)}
\makeatother

\makeatletter
\@namedef{subjclassname@2020}{%
  \textup{2020} Mathematics Subject Classification}
\makeatother

\begin{document}

\title[Diagonal hypersurfaces with monomial deformation]{The number of $\mathbb{F}_q$-points on diagonal\\hypersurfaces with monomial deformation}

\author{Dermot M\lowercase{c}Carthy}
\address{Dermot M\lowercase{c}Carthy, Department of Mathematics \& Statistics\\
Texas Tech University\\
Lubbock, TX 79410-1042\\
USA}
\email{dermot.mccarthy@ttu.edu}


\begin{abstract}
We consider the family of diagonal hypersurfaces with monomial deformation
$$D_{d, \lambda, h}: x_1^d  + x_2^d \dots + x_n^d - d \lambda \, x_1^{h_1} x_2^{h_2} \dots x_n^{h_n}=0$$
where $d = h_1+h_2 +\dots + h_n$ with $\gcd(h_1, h_2, \dots h_n)=1$.
We first provide a formula for the number of $\mathbb{F}_{q}$-points on $D_{d, \lambda, h}$ in terms of Gauss and Jacobi sums. This generalizes a result of Koblitz, which holds in the special case ${d \mid {q-1}}$. 
We then express the number of $\mathbb{F}_{q}$-points on $D_{d, \lambda, h}$ in terms of a $p$-adic hypergeometric function previously defined by the author. The parameters in this hypergeometric function mirror exactly those described by Koblitz when drawing an analogy between his result and classical hypergeometric functions. 
This generalizes a result by Sulakashna and Barman, which holds in the case $\gcd(d,{q-1})=1$.
In the special case $h_1 = h_2 = \dots =h_n = 1$ and $d=n$, i.e., the Dwork hypersurface, we also generalize a previous result of the author which holds when $q$ is prime.
\end{abstract}

\maketitle


\section{Introduction}\label{sec_Intro}
Counting the number of solutions to equations over finite fields using character sums dates back to the works of Gauss and Jacobi. A renewed interest in such problems followed subsequent important contributions from Hardy and Littlewood \cite{HL} and Davenport and Hasse \cite{DH}. In his seminal 1949 paper, Weil  \cite{W} gives an exposition on the topic up to that point (as well as going on to make his famous conjectures on the zeta functions of algebraic varieties). Specifically, he develops a formula for the number of solutions over $\mathbb{F}_q$, the finite field with $q$ elements, and its extensions, of $a_0 x_0^{n_0} + a_1 x_1^{n_1} + \dots +  a_k x_k^{n_k} =0$, in terms of what we now call Gauss sums and Jacobi sums. The techniques involved have since become standard practice and can be found in many well-known text books, e.g. \cite{BEW, IR}. Since then, many authors have used and adapted the techniques outlined in Weil's paper to study other equations, e.g. \cite{De, FG, Ko4}. Of particular interest is the work of Koblitz \cite{Ko4} where he examines the family of diagonal hypersurfaces with monomial deformation
\begin{equation}\label{def_Ddlh}
D_{d, \lambda, h}: x_1^d  + x_2^d \dots + x_n^d - d \lambda \, x_1^{h_1} x_2^{h_2} \dots x_n^{h_n}=0
\end{equation}
where $h_i \in \mathbb{Z}^{+}$, with $\gcd(h_1, h_2, \dots h_n)=1$, and $d = h_1+h_2 \dots + h_n$. Koblitz's main result \cite[Thm.~2]{Ko4} gives a formula for the number of $\mathbb{F}_q$-points on $D_{d, \lambda, h}$ in the terms of Gauss and Jacobi sums, in the case ${d \mid q-1}$. 
Using the analogy between Gauss sums and the gamma function, he notes that the main term in his formula can be considered a finite field analogue of a classical hypergeometric function. The purpose of this paper is to study $D_{d, \lambda, h}$ more generally, i.e., when the condition ${d \mid {q-1}}$ is removed. Firstly, we generalize Koblitz's result and provide a formula for the number of $\mathbb{F}_{q}$-points on $D_{d, \lambda, h}$ in terms of Gauss and Jacobi sums without the condition ${d \mid {q-1}}$. We then express the number of $\mathbb{F}_{q}$-points on $D_{d, \lambda, h}$ in terms of a $p$-adic hypergeometric function previously defined by the author. The parameters in this hypergeometric function mirror exactly those described by Koblitz when drawing an analogy between his result and classical hypergeometric functions. This generalizes a result of Sulakashna and Barman \cite{SB}, which holds in the case $\gcd(d,{q-1})=1$.
We also examine the special case when $h_1= h_2 = \dots h_n = 1$ and $d=n$, i.e., the Dwork hypersurface, and generalize a previous result of the author, which holds when $q$ is prime.


 \section{Statement of Results}\label{sec_Results}
Let $q=p^r$ be a prime power and let $\mathbb{F}_q$ denote the finite field with $q$ elements. 
Let $\widehat{\mathbb{F}^{*}_{q}}$ denote the group of multiplicative characters of $\mathbb{F}^{*}_{q}$.
We extend the domain of $\chi \in \widehat{\mathbb{F}^{*}_{q}}$ to $\mathbb{F}_{q}$ by defining $\chi(0):=0$ (including for the trivial character $\varepsilon$) and denote $\bar{\chi}$ as the inverse of $\chi$. Let $T$ be a fixed generator of $\widehat{\mathbb{F}_q^{*}}$.
Let $\theta$ be a fixed non-trivial additive character of $\mathbb{F}_q$ and for $\chi \in \widehat{\mathbb{F}^{*}_{q}}$ we define the Gauss sum $g(\chi):= \sum_{x \in \mathbb{F}_q} \chi(x) \theta(x)$. For $\chi_1, \chi_2, \dotsc, \chi_k \in \widehat{\mathbb{F}^{*}_{q}}$, we define the Jacobi sum
$J(\chi_1, \chi_2, \dotsc, \chi_k):= \sum_{t_i \in \mathbb{F}_q, t_1+t_2+ \dotsm + t_k=1} \chi_1(t_1) \chi_2(t_2) \dotsm \chi_k(t_k).$

We consider the family of diagonal hypersurfaces with monomial deformation described in (\ref{def_Ddlh}).
Let $t:=\gcd(d, q-1)$
and define
\begin{equation}\label{Def_W}
W:=\{w=(w_1, w_2, \ldots, w_n) \in \mathbb{Z}^n : 0 \leq w_i < t, \sum_{i=1}^n w_i \equiv 0 \pmod t\}. 
\end{equation} 
Define an equivalence relation $\sim_h$ on $W$ by 
\begin{equation}\label{Def_tilde}
w \sim_h w^{\prime} \textup{ if } w- w^{\prime} \textup{ is a multiple modulo $t$ of $h=(h_1,h_2, \ldots, h_n)$}. 
\end{equation}
If $h=(1,1, \dots, 1)$ we will write $\sim_1$.
We denote the class containing $w$ by $[w]$. We note also that each class contains a representative $w$ where some $w_i = 0$, for $1 \leq i \leq n$. We will write $[w_{_0}]$ to indicate that we have chosen such a representative for a particular class.

Our first result provides a formula for the number of $\mathbb{F}_{q}$-points on $D_{d, \lambda, h}$ in terms of Gauss and Jacobi sums, without the condition ${d \mid q-1}$. We will use $\mathbb{A}^{n}(\mathbb{F}_q)$ and $\mathbb{P}^{n}(\mathbb{F}_q)$ to denote the affine and projective $n$-spaces, respectively, over $\mathbb{F}_q$. 
We denote the subset of elements in these spaces where all co-ordinates are non-zero by $\mathbb{A}^{n}(\mathbb{F}_q^{*})$ and $\mathbb{P}^{n}(\mathbb{F}_q^{*})$.
\begin{theorem}\label{thm_2}
Let $N_q(D_{d, \lambda, h})$ be the number of points in $\mathbb{P}^{n-1}(\mathbb{F}_q)$ on $D_{d, \lambda, h}$. Then
\begin{multline*}
N_q(D_{d, \lambda, h}) 
= \frac{q^{n-1}-1}{q-1}
- \sum_{w^{*}} J(T^{w_1 \frac{q-1}{t}}, T^{w_2 \frac{q-1}{t}}, \dots, T^{w_n \frac{q-1}{t}}) \\
+\frac{1}{q-1} \sum_{s, w} \frac{g(T^{w_1 \frac{q-1}{t} + h_1 s}) g(T^{w_2 \frac{q-1}{t} + h_2 s}) \dots g(T^{w_n \frac{q-1}{t} + h_n s})}{g(T^{ds})} \, T^{ds}(d \lambda)
\end{multline*}
where the first sum is over all $w^{*}=(w_1, w_2, \dots, w_n) \in W$ such that $0 < w_i < t$ for all $i$, and the second sum is over all $s \in \{0,1,\dots, \frac{q-1}{t}-1\}$ and all $w=(w_1, w_2, \dots, w_n) \in W$.
\end{theorem}

\noindent  
Theorem \ref{thm_2} generalizes Koblitz \cite[Thm.~2]{Ko4}, which holds in the case $d \mid {q-1}$. Using an analogy between Gauss sums and the gamma function, Koblitz notes that the second summand in his formula, which corresponds to the second summand in Theorem \ref{thm_2} above with $t=d$, can be considered a finite field analogy of the classical hypergeometric function
\begin{equation}\label{for_KoblitzHypFn}
\prod_{i=1}^{n}   \Gamma(\tfrac{w_i}{d}) \cdot
 {_{d}F_{d-1}}\biggl[ \begin{array}{ccccc} \multicolumn{5}{c}{\dotsc \dotsc \frac{w_i}{dh_i} + \frac{b_i}{h_i} \dotsc \dotsc} \\[4pt]
 \phantom{1} & \frac{1}{d} & \frac{2}{d} & \dotsc & \frac{d-1}{d} \end{array}
\Big| \; \lambda^d  h_1^{h_1}  \dots h_n^{h_n} \; \biggr]
\end{equation}
where the top line parameters range over all $i=1, \dots, n$ and, for each $i$, all $b_i=0, \dots h_i-1$.
The main purpose of this paper is to express $N_q(D_{d, \lambda, h})$ in terms of a $p$-adic hypergeometric function previously defined by the author, whereby the parameters in this $p$-adic hypergeometric function mirror exactly those described by Koblitz in (\ref{for_KoblitzHypFn}) above.

Next, we rewrite Theorem \ref{thm_2} in a way more amenable to manipulation when we pass to the to the $p$-adic setting.
\begin{cor}\label{cor_thm2}
\begin{multline*}
N_q(D_{d, \lambda, h}) 
= \frac{q^{n-1}-1}{q-1}
-\frac{1}{q} \sum_{\substack{w \in W\\ \text{some $w_i=0$}}} \prod_{i=1}^{n} g(T^{w_i \frac{q-1}{t}}) \\
+\frac{1}{q(q-1)} \sum_{s, w} \prod_{i=1}^{n} g(T^{w_i \frac{q-1}{t} + h_i s}) {g(T^{-ds})} \, T^{ds}(-d \lambda)
\end{multline*}
where the first sum is over all $w=(w_1, w_2, \dots, w_n) \in W$ such that at least one $w_i=0$, and the second sum is over \textbf{either} all $s \in \{0,1,\dots, \frac{q-1}{t}-1\}$ and all $w=(w_1, w_2, \dots, w_n)\in W$ \textbf{or} all $s \in \{0,1,\dots, q-2\}$ and all $w=(w_1, w_2, \dots, w_n)\in W/\sim_h$. In the latter case, the sum is independent of the choice of equivalence class representatives.
\end{cor}

We now define our $p$-adic hypergeometric function. 
Let $\mathbb{Z}_p$ denote the ring of $p$-adic integers, $\mathbb{Q}_p$ the field of $p$-adic numbers, $\bar{\mathbb{Q}_p}$ the algebraic closure of $\mathbb{Q}_p$, and $\mathbb{C}_p$ the completion of $\bar{\mathbb{Q}_p}$. 
Let $\mathbb{Z}_q$ be the ring of integers in the unique unramified extension of $\mathbb{Q}_p$ with residue field $\mathbb{F}_q$.
Recall that for each $x \in \mathbb{F}_q^{*}$, there is a unique Teichm\"{u}ller representative ${\omega(x) \in \mathbb{Z}^{\times}_q}$ such that $\omega(x)$ is a $({q-1})$-st root of unity and $\omega(x) \equiv x \pmod p$. Therefore, we define the Teichm\"{u}ller character to be the primitive character $\omega:  \mathbb{F}_q^{*}  \rightarrow\mathbb{Z}^{\times}_q$ given by $x \mapsto \omega(x)$, which we extend with $\omega(0):=0$.

\begin{defi}\cite[Definition 5.1]{McC7}\label{def_Gq}
Let $q=p^r$ for $p$ an odd prime. Let $\lambda \in \mathbb{F}_q$, $m \in \mathbb{Z}^{+}$ and, for $1 \leq i \leq m$, $a_i, b_i \in \mathbb{Q} \cap \mathbb{Z}_p$.
Then define  
\begin{multline*}
{_{m}G_{m}}
\biggl[ \begin{array}{cccc} a_1, & a_2, & \dotsc, & a_m \\
 b_1, & b_2, & \dotsc, & b_m \end{array}
\Big| \; \lambda \; \biggr]_q
: = \frac{-1}{q-1}  \sum_{s=0}^{q-2} 
(-1)^{sm}\;
\bar{\omega}^s(\lambda)\\
\times \prod_{i=1}^{m} 
\prod_{k=0}^{r-1} 
\frac{\biggfp{\langle (a_i -\frac{s}{q-1} )p^k \rangle}}{\biggfp{\langle a_i p^k \rangle}}
\frac{\biggfp{\langle (-b_i +\frac{s}{q-1}) p^k \rangle}}{\biggfp{\langle -b_i p^k\rangle}}
(-p)^{-\lfloor{\langle a_i p^k \rangle -\frac{s p^k}{q-1}}\rfloor -\lfloor{\langle -b_i p^k\rangle +\frac{s p^k}{q-1}}\rfloor}.
\end{multline*}
\end{defi}
\noindent
We note that the value of ${_{m}G_{m}}[\cdots]$ depends only on the fractional part of the $a_i$ and $b_i$ parameters, and is invariant if we change the order of the parameters. Our main result expresses $N_q(D_{d, \lambda, h})$ in terms of this function.

\begin{theorem}\label{thm_Main}
Let $q=p^r$ for $p$ an odd prime.
Then, for $p \nmid dh_1 \dots h_n$,
\begin{multline*}
N_q(D_{d, \lambda, h}) 
= \frac{q^{n-1}-1}{q-1}
-\frac{(-1)^n}{q} \sum_{\substack{w \in W \\ \text{some $w_i=0$}}}
C(w)\\
\qquad \qquad 
+\frac{(-1)^n}{q} \sum_{[w] \in W/\sim_h} 
C(w) \;
{_{d}G_{d}}\biggl[ \begin{array}{ccccc} \multicolumn{5}{c}{\dotsc \dotsc \frac{w_i}{th_i} + \frac{b_i}{h_i} \dotsc \dotsc} \\[4pt]
 1 & \frac{1}{d} & \frac{2}{d} & \dotsc & \frac{d-1}{d} \end{array}
\Big| \; \bigl(\lambda^d  h_1^{h_1}  \dots h_n^{h_n} \bigr)^{-1} \; \biggr]_q
\end{multline*}
where the top line parameters in ${_{d}G_{d}}$ are the list $\left[ \frac{w_i}{th_i} + \frac{b_i}{h_i} \mid i=1, \dots, n; b_i=0,1,\dots h_i-1\right]$ and
\begin{equation}\label{for_Cw}
C(w) :=  \prod_{i=1}^{n}  \prod_{a=0}^{r-1}  \biggfp{\langle (\tfrac{w_i}{t}) p^a \rangle} (-p)^{\langle (\tfrac{w_i}{t} ) p^a \rangle}.
\end{equation}
\end{theorem}

\noindent As we can see, the parameters of ${_{d}G_{d}}$ in Theorem \ref{thm_Main} mirror exactly those in (\ref{for_KoblitzHypFn}) (when ${d \mid q-1}$ and so $t=d$) up to inversion of the argument $\lambda^d  h_1^{h_1}  \dots h_n^{h_n}$. This inversion is a feature of the definition of the function ${_{m}G_{m}}$. Because we are summing over $W/\sim_h$, we can remove this inversion while also swapping the top and bottom line parameters, which gives a more natural representation, in the opinion of the author. This can be seen more clearly later, in Corollary \ref{cor_DworkMain2}, where we get an all integral bottom line parameters.

\begin{cor}\label{cor_Main1}
Let $q=p^r$ for $p$ an odd prime. Then, for $p \nmid dh_1 \dots h_n$,
\begin{multline*}
N_q(D_{d, \lambda, h}) 
= \frac{q^{n-1}-1}{q-1}
-\frac{(-1)^n}{q}  \sum_{\substack{w \in W \\ \text{some $w_i=0$}}}
C(w)\\
+\frac{(-1)^n}{q} \sum_{[w] \in W/\sim_h}
C(-w) \;
{_{d}G_{d}}\biggl[ \begin{array}{ccccc} 1 & \frac{1}{d} & \frac{2}{d} & \dotsc & \frac{d-1}{d} \\[4pt]
 \multicolumn{5}{c}{\dotsc \dotsc \frac{w_i}{th_i} + \frac{b_i}{h_i} \dotsc \dotsc} \end{array}
\Big| \; \lambda^d  h_1^{h_1}  \dots h_n^{h_n}  \; \biggr]_q.
\end{multline*}
\end{cor}

Ideally, in Theorem \ref{thm_Main} and Corollary \ref{cor_Main1}, we would like to combine both sums into a single hypergeometric term. In general, it seems that this is not possible. However, it can be achieved in two special cases as we see in the next two results. The first is when $\gcd(d, q-1)=1$ and the second is when all $h_i=1$, i.e., the Dwork hypersurface.

\begin{cor}\label{cor_Main2}
Let $q=p^r$ for $p$ an odd prime. If $\gcd(d, q-1)=1$ then, for $p \nmid dh_1 \dots h_n$,
\begin{equation*}
N_q(D_{d, \lambda, h}) 
= \frac{q^{n-1}-1}{q-1}
+(-1)^n
{_{d-1}G_{d-1}}\biggl[ \begin{array}{cccc} \frac{1}{d} & \frac{2}{d} & \dotsc & \frac{d-1}{d} \\[4pt]
 \multicolumn{4}{c}{\dotsc \frac{b_i}{h_i}  \dotsc} \end{array}
\Big| \; \lambda^d  h_1^{h_1}  \dots h_n^{h_n}  \; \biggr]_q
\end{equation*}
where the bottom line parameters in ${_{d-1}G_{d-1}}$ are the list $[\frac{b_i}{h_i} \mid i=1, \dots, n; b_i=0,1,\dots h_i-1]$ with exactly one zero removed.
\end{cor}
\noindent Corollary \ref{cor_Main2} is Theorem 1.2 of \cite{SB}.

When $h_1= h_2 = \dots h_n = 1$ and $d=n$ in (\ref{def_Ddlh}), we recover the Dwork hypersurface, which we will denote $D_{\lambda}$, i.e.,
\begin{equation*}\label{def_Dl}
D_{\lambda}: x_1^n  + x_2^n \dots + x_n^n - n \lambda \, x_1 x_2 \dots x_n=0.
\end{equation*}
We now provide formulas for the number of $\mathbb{F}_{q}$-points on $D_{\lambda}$, first in terms of Gauss and Jacobi sums, and then in terms of the $p$-adic hypergeometric function.
For a given $w=(w_1, w_2, \ldots, w_n) \in W$, define $n_k$ to be the number of $k$'s appearing in $w$, i.e.,
$n_k = |\{w_i \mid  1 \leq i \leq n, w_i=k \} |.$
We then let $S_w := \{ k \mid 0 \leq k \leq {t-1}, n_k=0 \}$ and $S_w^{c}$ denote its complement in $\{0,1 \cdots, {t-1}\}$. So the elements of $S_w$ are the numbers from $0$ to ${t-1}$, inclusive, which do not appear in $w$.
We define the following lists
\begin{equation}\label{def_Aw}
A_w: \left[ \tfrac{t-k}{t} \mid k \in S_w \right] \cup \left[ \tfrac{b}{n} \mid 0 \leq b \leq n-1, b \not\equiv 0 \imod{\tfrac{n}{t}} \right];
\end{equation}
\begin{equation}\label{def_Bw}
B_w: \left[  \tfrac{t-k}{t} \, \textup{repeated $n_k$-1 times} \mid k \in S_w^c \right] .
\end{equation}
We note both lists contain 
$n-|S_w^c|$
numbers.

\begin{cor}[Corollary to Theorem \ref{thm_2}]\label{cor_thm2_Dwork}
Let $N_q(D_{\lambda})$ be the number of points in $\mathbb{P}^{n-1}(\mathbb{F}_q)$ on $D_{\lambda}$. Let $t=\gcd(n,q-1)$. Then, for $\lambda \neq0$,
\begin{multline*}
N_q(D_{\lambda}) 
= \frac{q^{n-1}-1}{q-1}\\
+\frac{1}{q(q-1)}\sum_{s, w} 
\left[\prod_{k \in S_w^c} \frac{g(T^{k \frac{q-1}{t} + s})^{n_k-1}}{g(T^{-k \frac{q-1}{t} - s})} T^{k \frac{q-1}{t}+s}(-1) \, q \right]
{g(T^{-ns})} \, T^{ns}(-n \lambda).
\end{multline*}
where the sum is over \textbf{either} all $s \in \{0,1,\dots, \frac{q-1}{t}-1\}$ and all $w=(w_1, w_2, \dots, w_n)\in W$ \textbf{or} all $s \in \{0,1,\dots, q-2\}$ and all $w=(w_1, w_2, \dots, w_n)\in W/\sim_1$. In the latter case, the sum is independent of the choice of equivalence class representatives.
\end{cor}

\begin{theorem}\label{thm_Dwork}
Let $q=p^r$ for $p$ an odd prime. Let $N_q(D_{\lambda})$ be the number of points in $\mathbb{P}^{n-1}(\mathbb{F}_q)$ on
$D_{\lambda}$
for some $\lambda \in \mathbb{F}_q^{*}$. Let $t=\gcd(n,q-1)$
and let $C(w)$ be defined by (\ref{for_Cw}). Then, for $p \nmid n$,
\begin{equation*}
N_q(D_{\lambda}) 
= \frac{q^{n-1}-1}{q-1}
+(-1)^n \sum_{[w_{_0}] \in W/\sim_1} 
C(w_{_0}) \;
{_{l}G_{l}}
\biggl[ \begin{array}{c} A_{w_{_0}} \\
 B_{w_{_0}} \end{array}
\Big| \; \lambda^n \; \biggr]_q.
\end{equation*}
\end{theorem}

\noindent Theorem \ref{thm_Dwork} generalizes Theorem 2.2 in \cite{McC12} which holds for $q=p$.
Finally, if we let $\gcd(d, q-1)=1$ in Theorem \ref{thm_Dwork}, or we let $h_1=h_2=\dots h_n=1$ in Corollary \ref{cor_Main2}, it easy to see that we arrive at the following result.
\begin{cor}\label{cor_DworkMain2}
If $\gcd(d, q-1)=1$ then, for $p \nmid n$,
\begin{equation*}
N_q(D_{\lambda}) 
= \frac{q^{n-1}-1}{q-1}
+(-1)^n
{_{n-1}G_{n-1}}\biggl[ \begin{array}{cccc} \frac{1}{n} & \frac{2}{n} & \dotsc & \frac{n-1}{n} \\[4pt]
1 & 1 & \dotsc & 1 \end{array}
\Big| \; \lambda^n  \; \biggr]_q.
\end{equation*}
\end{cor}
\noindent Corollary \ref{cor_DworkMain2} generalizes Corollary 2.3 in \cite{McC12} which holds for $q=p$.


\section{Preliminaries}\label{sec_Prelim}
We start by recalling some properties of Gauss and Jacobi sums. See \cite{BEW, IR} for further details, noting that we have adjusted results to take into account $\varepsilon(0)=0$, where $\varepsilon$ is the trivial character. We first note that $G(\varepsilon)=-1$. For $\chi \in \widehat{\mathbb{F}^{*}_{q}}$,
\begin{equation}\label{for_GaussConj}
G(\chi)G(\bar{\chi})=
\begin{cases}
\chi(-1) q & \chi \neq \varepsilon,\\
1 & \chi= \varepsilon.
\end{cases}
\end{equation}
For $\chi_1, \chi_2, \dotsc, \chi_k \in \widehat{\mathbb{F}^{*}_{q}}$ and $\alpha \in \mathbb{F}_q$, we define the generalized Jacobi sum
$$J_{\alpha}(\chi_1, \chi_2, \dotsc, \chi_k):= \sum_{t_i \in \mathbb{F}_q, t_1+t_2+ \dotsm + t_k=\alpha} \chi_1(t_1) \chi_2(t_2) \dotsm \chi_k(t_k).$$
When $\alpha=1$ we recover the usual Jacobi sum as defined in Section \ref{sec_Results}.
\begin{prop}\label{prop_Jac0}
For $\chi_1, \chi_2, \dotsc, \chi_k \in \widehat{\mathbb{F}^{*}_{q}}$,
\begin{equation*}
J_{0}(\chi_1, \chi_2, \dotsc, \chi_k)=
\begin{cases}
(q-1)^k - (q-1) \: J(\chi_1, \chi_2, \dotsc, \chi_k) & \quad \textup{if } \chi_1, \chi_2, \dotsc, \chi_k \text{ all trivial,}\\[9pt] 
-(q-1) J(\chi_1, \chi_2, \dotsc, \chi_k) & \quad \textup{if } \chi_1\chi_2\dotsm\chi_k \text{ trivial but at least}\\
&\quad\text{one of } \chi_1, \chi_2, \dotsc, \chi_k \text{ non-trivial,}\\[9pt]
0 &\quad \textup{if } \chi_1\chi_2\dotsm\chi_k \text{ non-trivial.}
\end{cases}
\end{equation*}
\end{prop}

\begin{prop}\label{prop_JacRedSmall}
For $\chi_1\chi_2\dotsm\chi_k$ trivial but at least one of $\chi_1, \chi_2, \dotsc, \chi_k$ non-trivial then
\begin{equation*}
 J(\chi_1, \chi_2, \dotsc, \chi_k)=-\chi_k(-1)  J(\chi_1, \chi_2, \dotsc, \chi_{k-1}) \; .
 \end{equation*}
\end{prop}

\begin{prop}\label{prop_JacRedAllTrivial}
For $\chi_1, \chi_2, \dotsc, \chi_k$ all trivial,
\begin{equation*}
J(\chi_1, \chi_2, \dotsc, \chi_k) =[(q-1)^k+(-1)^{k+1}]/q
\end{equation*}
\end{prop}

\begin{prop}\label{prop_JactoGauss}
For $\chi_1, \chi_2, \dotsc, \chi_k$ not all trivial,
\begin{equation*}
J(\chi_1, \chi_2, \dotsc, \chi_k)=
\begin{cases}
\dfrac{G(\chi_1)G(\chi_2)\dotsc G(\chi_k)}{G(\chi_1 \chi_2 \dotsm \chi_k)}
& \qquad \chi_1 \chi_2 \dotsm \chi_k \neq \varepsilon,\\[18pt]
-\dfrac{G(\chi_1)G(\chi_2)\dotsc G(\chi_k)}{q}
&\qquad \chi_1 \chi_2 \dotsm \chi_k = \varepsilon .
\end{cases}
\end{equation*}
\end{prop}

We now recall the $p$-adic gamma function. For further details, see \cite{Ko}.
Let $p$ be an odd prime.  For $n \in \mathbb{Z}^{+}$ we define the $p$-adic gamma function as
\begin{align*}
\gfp{n} &:= {(-1)}^n \prod_{\substack{0<j<n\\p \nmid j}} j 
\end{align*}
and extend it to all $x \in\mathbb{Z}_p$ by setting $\gfp{0}:=1$ and
$\gfp{x} := \lim_{n \rightarrow x} \gfp{n}$
for $x\neq 0$, where $n$ runs through any sequence of positive integers $p$-adically approaching $x$. 
This limit exists, is independent of how $n$ approaches $x$, and determines a continuous function
on $\mathbb{Z}_p$ with values in $\mathbb{Z}^{*}_p$.
The function satisfies the following product formula.
\begin{theorem}[Gross, Koblitz \cite{GK} Thm.~3.1]\label{thm_GrossKoblitzMult}
If $h\in\mathbb{Z}^{+}$, $p \nmid h$ and $0\leq x <1$ with $(q-1)x \in \mathbb{Z}$, then
\begin{equation}\label{for_pGammaMult}
\prod_{a=0}^{r-1} \prod_{b=0}^{h-1} \gfp{\langle \tfrac{x+b}{h} p^a \rangle} 
= \omega \big(h^{(q-1)x}\bigr)
\prod_{a=0}^{r-1} \gfp{\langle x p^a \rangle} \prod_{b=1}^{h-1} \gfp{\langle \tfrac{b}{h} p^a \rangle}.
\end{equation}
\end{theorem}
\noindent We note that in the original statement of Theorem \ref{thm_GrossKoblitzMult} in \cite{GK}, $\omega$ is the Teichm\"{u}ller character of $\mathbb{F}_p^{*}$. However, the result above still holds as $\omega|_{\mathbb{F}_p^{*}}$ is the Teichm\"{u}ller character of $\mathbb{F}_p^{*}$.

The Gross-Koblitz formula allows us to relate Gauss sums and the $p$-adic gamma function. Let $\pi \in \mathbb{C}_p$ be the fixed root of $x^{p-1}+p=0$ that satisfies ${\pi \equiv \zeta_p-1 \pmod{{(\zeta_p-1)}^2}}$.
\begin{theorem}[Gross, Koblitz \cite{GK} Thm.~1.7]\label{thm_GrossKoblitz}
For $ j \in \mathbb{Z}$,
\begin{equation*} 
g(\bar{\omega}^j)=-\pi^{(p-1) \sum_{a=0}^{r-1} \langle{\frac{jp^a}{q-1}}\rangle} \: \prod_{a=0}^{r-1} \gfp{\langle{\tfrac{jp^a}{q-1}}\rangle}.
\end{equation*}
\end{theorem}

We now recall some results of Weil \cite{W} and Koblitz \cite{Ko4}. Note that the definitions and notation used for characters and for Gauss and Jacobi sums in both those papers differ from each other, and differ from what's defined in this paper. So, we have adjusted their results accordingly. 
For $d \in \mathbb{Z}^{+}$, let $D_{d}$ denote the diagonal hypersurface
$$D_{d}: x_1^d  + x_2^d \dots + x_n^d = 0.$$

\begin{theorem}[Weil \cite{W}]\label{thm_Weil}
Let $N_q^{A}(D_{d}) $ be the number of points in $\mathbb{A}^{n}(\mathbb{F}_q)$ on 
$D_{d}$.
Let $t:=\gcd(d,q-1)$. Then
$$N_q^{A}(D_{d}) = q^{n-1}-(q-1) \sum_{w^{*}} J(T^{w_1 \frac{q-1}{t}}, T^{w_2 \frac{q-1}{t}}, \dots, T^{w_n \frac{q-1}{t}})$$
where the sum is over all $w^{*}=(w_1, w_2, \dots, w_n) \in W$ such that $0 < w_i < t$.
\end{theorem}

\noindent Using similar methods to those in \cite{W} and \cite[Thm.~2]{Ko4} it is easy to see that 
\begin{theorem}\label{thm_WeilNonZero}
Let $N_q^{A, *}(D_{d}) $ be the number of points in $\mathbb{A}^{n}(\mathbb{F}_q^{*})$ on 
$D_{d}$
Let $t:=\gcd(d,q-1)$. Then
$$N_q^{A, *}(D_{d}) = \sum_{w \in W} J_0(T^{w_1 \frac{q-1}{t}}, T^{w_2 \frac{q-1}{t}}, \dots, T^{w_n \frac{q-1}{t}}).$$
\end{theorem}

\begin{theorem}[Koblitz \cite{Ko4} {Thm.~1}]\label{thm_1Koblitz}
Let $N_q^{A,*}$ be the number of points in $\mathbb{A}^{n}(\mathbb{F}_q^{*})$ on 
$$\sum_{i=1}^r a_i x_1^{m_{1i}} x_2^{m_{2i}} \dots x_n^{m_{ni}}=0$$
for some $a_i \in \mathbb{F}_q^{*}$, $m_{ji} \in \mathbb{Z}_{\geq0}$, such that for a given $i$, $m_{ji}$ are not all zero. 
Then
\begin{multline*}
N_q^{A,*} = \tfrac{1}{q} \left[ (q-1)^n + (-1)^r (q-1)^{n-r+1}\right] \\
-(q-1)^{n-r+1}
\sum_{\alpha} T^{-\alpha_1}(a_1) T^{-\alpha_2}(a_2) \dots T^{-\alpha_r}(a_r) J(T^{\alpha_1}, T^{\alpha_2}, \dots, T^{\alpha_r})
\end{multline*}
where the sum is over all $\alpha= (\alpha_1, \alpha_2, \dots, \alpha_r) \neq 0$ satisfying
$0 \leq \alpha_i < {q-1}$,
$\sum_{i=1}^{r} \alpha_i \equiv 0 \pmod{{q-1}}$,
and,
$\sum_{i=1}^{r} m_{ji} \alpha_i \equiv 0 \pmod{{q-1}}$ for all $j \in \{1, 2, \dots, n\}$.
\end{theorem}

\noindent An important first step in proving the main results of this paper is to adapt Theorem \ref{thm_1Koblitz} to $D_{d, \lambda, h}$ as follows.
\begin{cor}\label{cor_1}
Let $t:=\gcd(d,q-1)$. For $\lambda \neq 0$,
\begin{equation*}
N_q^{A,*}(D_{d, \lambda, h}) 
= \sum_{s, w} J(T^{w_1 \frac{q-1}{t} + h_1 s}, T^{w_2 \frac{q-1}{t} + h_2 s}, \dots, T^{w_n \frac{q-1}{t} + h_n s}) \, T^{ds}(d \lambda)
\end{equation*}
where the sum is over all $s \in \{0,1,\dots, \frac{q-1}{t}-1\}$ and all $w=(w_1, w_2, \dots, w_n) \in W$.
\end{cor}

\noindent Corollary \ref{cor_1} generalizes Corollary 1 in \cite{Ko4}, which holds in the case $d \mid {q-1}$.


\section{Proofs}\label{sec_Proofs}

\begin{proof}[Proof of Corollary \ref{cor_1}]
We take $r=n+1$; 
$a_i=1$, for $i=1, \dots, n$, and $a_r = -d \lambda$;
$m_{ji} = d$ if $i=j$ and zero otherwise,
and, $m_{jr} = h_j$,
for all $j=1, \dots, n$,
in Theorem \ref{thm_1Koblitz}. This yields
\begin{equation}\label{for_NAstarD1}
N_q^{A,*}(D_{d, \lambda, h}) 
=\tfrac{1}{q} \left[ (q-1)^n + (-1)^{n+1}\right] 
-\sum_{\alpha} T^{-\alpha_{n+1}}(- d \lambda) J(T^{\alpha_1}, T^{\alpha_2}, \dots, T^{\alpha_{n+1}})
\end{equation}
where the sum is over all $\alpha= (\alpha_1, \alpha_2, \dots, \alpha_{n+1})  \neq 0$ satisfying
$0 \leq \alpha_i < q-1$,
$\sum_{i=1}^{n+1} \alpha_i \equiv 0 \pmod{q-1}$,
and,
$d \, \alpha_j + h_j \, \alpha_{n+1} \equiv 0 \pmod{q-1}$ for all $j=1, 2, \dots, n.$

The condition $d \, \alpha_j + h_j \, \alpha_{n+1} \equiv 0 \pmod{q-1}$, for all $j \in \{1, 2, \dots, n\}$, implies $t=\gcd(d, q-1)$ divides $h_j \, \alpha_{n+1}$ for all $j  \in \{1, 2, \dots, n\}$. If $l^e$ is a prime power dividing $t$ but not $\alpha_{n+1}$, then $l$ divides $h_j$ for all $j \in \{1, 2, \dots, n\}$. This is a contradiction, as $\gcd(h_1, \dots, h_n)=1$. Therefore, $l^e$ divides $\alpha_{n+1}$, which implies $t$ divides $\alpha_{n+1}$. So $\frac{\alpha_{n+1}}{t} \in \{0, 1, \dots, \frac{q-1}{t}-1 \}$. Let 
$s \equiv - \left(\frac{d}{t}\right)^{-1} \frac{\alpha_{n+1}}{t} \pmod {\frac{q-1}{t}}$
such that $s \in \{0, 1, \dots, \frac{q-1}{t}-1 \}$. Then $s$ runs around $\{0, 1, \dots, \frac{q-1}{t}-1 \}$ as $\frac{\alpha_{n+1}}{t}$ does.

We now express the conditions on $\alpha$ in terms of $s$. Firstly,
\begin{align*} 
d \, \alpha_j &\equiv - h_j \, \alpha_{n+1}  \pmod{q-1}\\
\Rightarrow \, \tfrac{d}{t} \alpha_j &\equiv - h_j \tfrac{\alpha_{n+1}}{t}  \pmod{\tfrac{q-1}{t}}\\
\Rightarrow \, \alpha_j &\equiv  h_j s  \pmod{\tfrac{q-1}{t}}.
\end{align*}
So $\alpha_j = h_j s + w_j \frac{q-1}{t}$ for $w_j \in \{0,1, \dots, t-1\}$, for $j \in \{1, 2, \dots, n\}$. Also,
\begin{align}\label{for_Un+1}
\notag \tfrac{\alpha_{n+1}}{t} & \equiv - \left(\tfrac{d}{t}\right) s \pmod {\tfrac{q-1}{t}}\\
 \Rightarrow \,  \alpha_{n+1} & \equiv - ds \pmod{{q-1}}.
\end{align}
Using the fact that $\sum_{i=j}^{n} h_j=d$, it is easy to see that
\begin{equation}\label{for_sumwi}
\sum_{j=1}^{n} w_j = \sum_{j=1}^{n} \tfrac{t}{q-1} \left( \alpha_j - h_j s \right) =  \tfrac{t}{q-1}  \biggl( \sum_{j=1}^{n} \alpha_j - d s \biggr).
\end{equation}
Combining (\ref{for_Un+1}) and (\ref{for_sumwi}) we get that
$$\sum_{j=1}^{n} w_j \equiv 0 \pmod{t} \quad \Longleftrightarrow \quad  \sum_{i=1}^{n+1} \alpha_i \equiv 0 \pmod{q-1}.$$
Substituting for $\alpha$, (\ref{for_NAstarD1}) becomes
\begin{multline}\label{for_NAstarD2}
N_q^{A,*}(D_{d, \lambda, h}) 
=\tfrac{1}{q} \left[ (q-1)^n + (-1)^{n+1}\right] \\
-\sum_{s, w} T^{ds}(- d \lambda) J(T^{w_1 \frac{q-1}{t} + h_1 s}, \dots, T^{w_n \frac{q-1}{t} + h_n s}, T^{-ds})
\end{multline}
where the sum is over all $s \in \{0,1,\dots, \frac{q-1}{t}-1\}$ and all $w=(w_1, w_2, \dots, w_n)$, such that $0 \leq w_i < t$ and $\sum_{i=1}^{n} w_i \equiv 0 \pmod t$, and such that not all of $s, w_1, w_2, \dots w_n$ are zero.

Noting that, as $\sum_{i=1}^{n} w_i \frac{q-1}{t} + h_i s - ds \equiv 0 \pmod {q-1}$, by Proposition \ref{prop_JacRedSmall} we have
$$J(T^{w_1 \frac{q-1}{t} + h_1 s}, \dots, T^{w_n \frac{q-1}{t} + h_n s}, T^{-ds}) = -J(T^{w_1 \frac{q-1}{t} + h_1 s}, \dots, T^{w_n \frac{q-1}{t} + h_n s}) \, T^{-ds}(-1),$$
and by Proposition \ref{prop_JacRedAllTrivial} we have
$$J(\underbrace{T^0, T^0, \dots, T^0}_{n \; times}) = \tfrac{1}{q} \left[ (q-1)^n + (-1)^{n+1}\right]$$
completes the proof.
\end{proof}

\begin{proof}[Proof of Theorem \ref{thm_2}]
We follow Koblitz \cite[Thm.~2]{Ko} and note that
\begin{equation}\label{for_NsRelation}
N_q(D_{d, \lambda, h}) - N_q^{*}(D_{d, \lambda, h})  = N_q(D_{d, 0, h}) - N_q^{*}(D_{d, 0, h}).
\end{equation}
We know 
\begin{equation}\label{for_NqDd0h}
N_q(D_{d, 0, h}) = \frac{N_q^{A}(D_{d}) -1}{q-1} = \frac{q^{n-1}-1}{q-1}-\sum_{w^{*}} J(T^{w_1 \frac{q-1}{t}}, T^{w_2 \frac{q-1}{t}}, \dots, T^{w_n \frac{q-1}{t}})
\end{equation}
by Weil's result, Theorem \ref{thm_Weil} above;
$$N_q^{*}(D_{d, \lambda, h}) = \frac{1}{q-1} \sum_{s, w} J(T^{w_1 \frac{q-1}{t} + h_1 s}, T^{w_2 \frac{q-1}{t} + h_2 s}, \dots, T^{w_n \frac{q-1}{t} + h_n s}) \, T^{ds}(d \lambda)$$
when $\lambda \neq 0$, by Corollary \ref{cor_1}; and
$$N_q^{*}(D_{d, 0, h}) = N_q^{*}(D_{d}) = \frac{1}{q-1} \sum_{w} J_0(T^{w_1 \frac{q-1}{t}}, T^{w_2 \frac{q-1}{t}}, \dots, T^{w_n \frac{q-1}{t}})$$
by Theorem \ref{thm_WeilNonZero}.

Using Propositions \ref{prop_Jac0}, \ref{prop_JacRedAllTrivial} and \ref{prop_JactoGauss}, we get that for $\lambda \neq 0$,
\begin{align}\label{for_N0Diff}
\notag (q-&1) \left( N_q^{*}(D_{d, \lambda, h}) - N_q^{*}(D_{d, 0, h}) \right)\\
\notag &= \sum_{\substack{s, w\\s \neq 0}} J(T^{w_1 \frac{q-1}{t} + h_1 s}, T^{w_2 \frac{q-1}{t} + h_2 s}, \dots, T^{w_n \frac{q-1}{t} + h_n s}) \, T^{ds}(d \lambda)\\
\notag & \qquad + \sum_{w} J(T^{w_1 \frac{q-1}{t}}, T^{w_2 \frac{q-1}{t}}, \dots, T^{w_n \frac{q-1}{t}}) 
- \sum_{w} J_0(T^{w_1 \frac{q-1}{t}}, T^{w_2 \frac{q-1}{t}}, \dots, T^{w_n \frac{q-1}{t}})\\
\notag &= \sum_{\substack{s, w\\s \neq 0}} J(T^{w_1 \frac{q-1}{t} + h_1 s}, T^{w_2 \frac{q-1}{t} + h_2 s}, \dots, T^{w_n \frac{q-1}{t} + h_n s}) \, T^{ds}(d \lambda)\\
\notag & \qquad + q \sum_{\substack{w \\w \neq 0}} J(T^{w_1 \frac{q-1}{t}}, T^{w_2 \frac{q-1}{t}}, \dots, T^{w_n \frac{q-1}{t}}) 
+ q J(\varepsilon, \varepsilon, \dots, \varepsilon) - (q-1)^n\\
\notag &= \sum_{\substack{s, w\\s \neq 0}} J(T^{w_1 \frac{q-1}{t} + h_1 s}, T^{w_2 \frac{q-1}{t} + h_2 s}, \dots, T^{w_n \frac{q-1}{t} + h_n s}) \, T^{ds}(d \lambda)\\
\notag & \qquad + q \sum_{\substack{w \\w \neq 0}} J(T^{w_1 \frac{q-1}{t}}, T^{w_2 \frac{q-1}{t}}, \dots, T^{w_n \frac{q-1}{t}}) 
+ (-1)^{n+1}\\
&= \sum_{s, w} \frac{g(T^{w_1 \frac{q-1}{t} + h_1 s}) g(T^{w_2 \frac{q-1}{t} + h_2 s}) \dots g(T^{w_n \frac{q-1}{t} + h_n s})}{g(T^{ds})} \, T^{ds}(d \lambda).
\end{align}
Combining (\ref{for_NsRelation}), (\ref{for_NqDd0h}) and (\ref{for_N0Diff}), which trivially holds for $\lambda =0$ also, yields the result.
\end{proof}

\begin{proof}[Proof of Corollary \ref{cor_thm2}]

Applying  (\ref{for_GaussConj}) and Proposition \ref{prop_JactoGauss} to Theorem \ref{thm_2} we get that
\begin{align*}
N_q(D_{d, \lambda, h}) 
&= \frac{q^{n-1}-1}{q-1}
+ \frac{1}{q} \sum_{w^{*}} \prod_{i=1}^{n} g(T^{w_i \frac{q-1}{t}})
-\frac{1}{q-1} \sum_{w} \prod_{i=1}^{n} g(T^{w_i \frac{q-1}{t}}) \\
&\qquad \qquad +\frac{1}{q(q-1)} \sum_{\substack{s, w \\ s \neq 0}} \prod_{i=1}^{n} g(T^{w_i \frac{q-1}{t} + h_i s}) {g(T^{-ds})} \, T^{ds}(-d \lambda)\\
&= \frac{q^{n-1}-1}{q-1}
+ \frac{1}{q} \sum_{w^{*}} \prod_{i=1}^{n} g(T^{w_i \frac{q-1}{t}})
-\frac{1}{q-1} \sum_{w} \prod_{i=1}^{n} g(T^{w_i \frac{q-1}{t}}) \left(1-\frac{1}{q} \right)\\
&\qquad \qquad +\frac{1}{q(q-1)} \sum_{s, w} \prod_{i=1}^{n} g(T^{w_i \frac{q-1}{t} + h_i s}) {g(T^{-ds})} \, T^{ds}(-d \lambda)\\
&= \frac{q^{n-1}-1}{q-1}
-\frac{1}{q} \sum_{\substack{w \\ \text{some $w_i=0$}}} \prod_{i=1}^{n} g(T^{w_i \frac{q-1}{t}}) \\
&\qquad \qquad+\frac{1}{q(q-1)} \sum_{s, w} \prod_{i=1}^{n} g(T^{w_i \frac{q-1}{t} + h_i s}) {g(T^{-ds})} \, T^{ds}(-d \lambda)
\end{align*}
where the last sum is over all $s \in \{0,1,\dots, \frac{q-1}{t}-1\}$ and all $w=(w_1, w_2, \dots, w_n)\in W$, as required.
To get the alternative summation limits, we note that
\begin{align}\label{for_ChangeSumLimits}
\notag \sum_{s=0}^{\frac{q-1}{t}-1} &\sum_{w \in W} \prod_{i=1}^{n} g(T^{w_i \frac{q-1}{t} + h_i s}) {g(T^{-ds})} \, T^{ds}(-d \lambda)\\
\notag &=\sum_{s=0}^{\frac{q-1}{t}-1} \sum_{j=0}^{t-1} \sum_{[w] \in W/\sim} g(T^{(w_i+j h_i) \frac{q-1}{t} + h_i s}) {g(T^{-ds})} \, T^{ds}(-d \lambda)\\
\notag &=\sum_{s=0}^{\frac{q-1}{t}-1} \sum_{j=0}^{t-1} \sum_{[w] \in W/\sim} g(T^{w_i \frac{q-1}{t} + h_i (s + j\frac{q-1}{t})}) {g(T^{-ds})} \, T^{ds}(-d \lambda)\\
&= \sum_{s=0}^{q-2} \sum_{[w] \in W/\sim} \prod_{i=1}^{n} g(T^{w_i \frac{q-1}{t} + h_i s}) {g(T^{-ds})} \, T^{ds}(-d \lambda).
\end{align}
This sum is independent of the choice of equivalence class representatives $[w]$, as changing representative can be countered by a simple change of variable in $s$.
\end{proof}

\begin{proof}[Proof of Corollary \ref{cor_thm2_Dwork}]
We start from Corollary \ref{cor_thm2} with $h=(1,1,\dots,1)$ and $d=n$, and re-write using the notation described in Section \ref{sec_Results}, i.e.,
\begin{multline}\label{for_DworkNdl1}
N_q(D_{\lambda}) 
= \frac{q^{n-1}-1}{q-1}
-\frac{1}{q} \sum_{\substack{w \in W \\ 0 \in S_w^c}} \prod_{k \in S_w^c} g(T^{k \frac{q-1}{t}})^{n_k} \\
+\frac{1}{q(q-1)} \sum_{s, w}  \prod_{k \in S_w^c}  g(T^{k \frac{q-1}{t} + s})^{n_k} {g(T^{-ns})} \, T^{ns}(-n \lambda)
\end{multline}
where $t=\gcd(n,q-1)$, and the second sum is over all $s \in \{0,1,\dots, \frac{q-1}{t}-1\}$ and all $w \in W$.
We proceed in the same fashion as the proof of Theorem 2.2 in \cite{McC12}. By (\ref{for_GaussConj}) it is easy to see that
\begin{equation}\label{for_DworkTerm1}
\sum_{\substack{w \\ 0 \in S_w^c}}  \prod_{k \in S_w^c} g(T^{k \frac{q-1}{t}})^{n_k}
=
\sum_{\substack{w \\ 0 \in S_w^c}} 
\left[\prod_{k \in S_w^c} \frac{g(T^{k \frac{q-1}{t}})^{n_k-1}}{g(T^{-k \frac{q-1}{t}})} \right]
\left[\prod_{k \in S_w^c\setminus\{0\}} T^{k \frac{q-1}{t}}(-1) \, q \right]
\end{equation}
We now focus on the second sum in (\ref{for_DworkNdl1}). If $T^{k \frac{q-1}{t} + s} = \varepsilon$ then $k \frac{q-1}{t} + s \equiv 0 \imod{q-1}$, which can only happen if $s \equiv 0 \imod{\frac{q-1}{t}}$, in which case $s=0$. So, if $s \neq 0$ then $T^{k \frac{q-1}{t} + s} \neq \varepsilon$.
Again using (\ref{for_GaussConj}), we see that, for $\lambda \neq 0$,
\begin{align}\label{for_DworkTerm2}
\notag
\sum_{w \in W} & \sum_{s=0}^{\frac{q-1}{t}}  \prod_{k \in S_w^c}  g(T^{k \frac{q-1}{t} + s})^{n_k} {g(T^{-ns})} \, T^{ns}(-n \lambda)\\
\notag &=
\sum_{w \in W} \sum_{s=1}^{\frac{q-1}{t}} 
\left[\prod_{k \in S_w^c} \frac{g(T^{k \frac{q-1}{t} + s})^{n_k-1}}{g(T^{-k \frac{q-1}{t} - s})} T^{k \frac{q-1}{t}+s}(-1) \, q \right]
{g(T^{-ns})} \, T^{ns}(-n \lambda)\\
\notag & \qquad - \sum_{w \in W}
\left[\prod_{k \in S_w^c} \frac{g(T^{k \frac{q-1}{t}})^{n_k-1}}{g(T^{-k \frac{q-1}{t}})} \right]
\left[\prod_{k \in S_w^c\setminus\{0\}} T^{k \frac{q-1}{t}}(-1) \, q \right]\\
\notag &=
\sum_{w \in W} \sum_{s=0}^{\frac{q-1}{t}} 
\left[\prod_{k \in S_w^c} \frac{g(T^{k \frac{q-1}{t} + s})^{n_k-1}}{g(T^{-k \frac{q-1}{t} - s})} T^{k \frac{q-1}{t}+s}(-1) \, q \right]
{g(T^{-ns})} \, T^{ns}(-n \lambda)\\
\notag & \qquad + \sum_{w \in W}
\left[\prod_{k \in S_w^c} \frac{g(T^{k \frac{q-1}{t}})^{n_k-1}}{g(T^{-k \frac{q-1}{t}})} \right]
\left[\prod_{k \in S_w^c} T^{k \frac{q-1}{t}}(-1) \, q - \prod_{k \in S_w^c\setminus\{0\}} T^{k \frac{q-1}{t}}(-1) \, q \right]\\
\notag &=
\sum_{w \in W} \sum_{s=0}^{\frac{q-1}{t}}  
\left[\prod_{k \in S_w^c} \frac{g(T^{k \frac{q-1}{t} + s})^{n_k-1}}{g(T^{-k \frac{q-1}{t} - s})} T^{k \frac{q-1}{t}+s}(-1) \, q \right]
{g(T^{-ns})} \, T^{ns}(-n \lambda)\\
& \qquad + (q-1)\sum_{\substack{w \\ 0 \in S_w^c}} 
\left[\prod_{k \in S_w^c} \frac{g(T^{k \frac{q-1}{t}})^{n_k-1}}{g(T^{-k \frac{q-1}{t}})} \right]
\left[\prod_{k \in S_w^c\setminus\{0\}} T^{k \frac{q-1}{t}}(-1) \, q \right].
\end{align}
Accounting for (\ref{for_DworkTerm1}) and (\ref{for_DworkTerm2}) in (\ref{for_DworkNdl1}) yields
\begin{multline*}
N_q(D_{\lambda}) 
= \frac{q^{n-1}-1}{q-1}\\
+\frac{1}{q(q-1)}\sum_{w \in W} \sum_{s=0}^{\frac{q-1}{t}} 
 \left[\prod_{k \in S_w^c} \frac{g(T^{k \frac{q-1}{t} + s})^{n_k-1}}{g(T^{-k \frac{q-1}{t} - s})} T^{k \frac{q-1}{t}+s}(-1) \, q \right]
{g(T^{-ns})} \, T^{ns}(-n \lambda).
\end{multline*}
To get the alternative summation limit, proceed in the same manner as in (\ref{for_ChangeSumLimits}).
\end{proof}

\begin{proof}[Proof of Theorem \ref{thm_Main}]
We start from Corollary \ref{cor_thm2}, which we rewrite as 
\begin{equation}\label{for_NqDdlh}
N_q(D_{d, \lambda, h}) 
= \frac{q^{n-1}-1}{q-1}
-\frac{1}{q} \sum_{\substack{w \in W \\ \text{some $w_i=0$}}} \prod_{i=1}^{n} g(T^{w_i \frac{q-1}{t}}) \\
+\frac{1}{q(q-1)} \sum_{[w] \in W/\sim} R_{[w]}
\end{equation}
where
$$R_{[w]}:=\sum_{s=0}^{q-2} \prod_{i=1}^{n} g(T^{w_i \frac{q-1}{t} + h_i s}) {g(T^{-ds})} \, T^{ds}(-d \lambda).$$
We note $R_{[w]}$ is independent of the choice of equivalence class representative.

We now let $T = \bar{\omega}$ and apply the Gross-Koblitz formula, Theorem \ref{thm_GrossKoblitz}, to both summands in (\ref{for_NqDdlh}).
From the first summand we get that
\begin{equation}\label{for_Summand1}
\prod_{i=1}^{n} g(T^{w_i \frac{q-1}{t}}) = 
(-1)^n 
(-p)^{  \sum_{i=1}^{n} \sum_{a=0}^{r-1} \langle (\tfrac{w_i}{t} ) p^a \rangle}
\prod_{i=1}^{n}  \prod_{a=0}^{r-1}  \biggfp{\langle (\tfrac{w_i}{t}) p^a \rangle} 
=(-1)^n \, C(w).
\end{equation}
The second, $R_{[w]}$, yields
\begin{equation}\label{for_Rw1}
R_{[w]} = (-1)^{n+1} \sum_{s=0}^{q-2} \left[ \prod_{a=0}^{r-1} \prod_{i=1}^{n} \biggfp{\langle (\tfrac{w_i}{t} +\tfrac{h_i s}{q-1}) p^a \rangle} \right] 
\left[\prod_{a=0}^{r-1} \biggfp{\langle (\tfrac{-ds}{q-1}) p^a \rangle} \right]  (-p)^{v} \; \bar{\omega}^{ds}(-d\lambda)
\end{equation}
where
\begin{align*}
v&=\sum_{a=0}^{r-1}  \sum_{i=1}^{n} \langle (\tfrac{w_i}{t} +\tfrac{h_i s}{q-1}) p^a \rangle + \sum_{a=0}^{r-1}\langle (\tfrac{-ds}{q-1}) p^a \rangle\\
&=  \sum_{a=0}^{r-1}  \sum_{i=1}^{n}  (\tfrac{w_i}{t} +\tfrac{h_i s}{q-1}) p^a + \sum_{a=0}^{r-1} (\tfrac{-ds}{q-1}) p^a
-\sum_{a=0}^{r-1}  \sum_{i=1}^{n} \left\lfloor  (\tfrac{w_i}{t} +\tfrac{h_i s}{q-1}) p^a \right\rfloor -\sum_{a=0}^{r-1} \left\lfloor  (\tfrac{-ds}{q-1}) p^a \right\rfloor\\
&=  \sum_{a=0}^{r-1}  \sum_{i=1}^{n}  (\tfrac{w_i}{t} ) p^a 
-\sum_{a=0}^{r-1}  \sum_{i=1}^{n} \left\lfloor  (\tfrac{w_i}{t} +\tfrac{h_i s}{q-1}) p^a \right\rfloor -\sum_{a=0}^{r-1} \left\lfloor  (\tfrac{-ds}{q-1}) p^a \right\rfloor \in \mathbb{Z}
\end{align*}
as $\sum_{i=1}^{n} h_i =d$ and $\sum_{i=1}^{n} w_i \equiv 0 \pmod{t}$.

We will now use Theorem \ref{thm_GrossKoblitzMult} to expand the terms involving the $p$-adic gamma function in (\ref{for_Rw1}). Let $k \in \mathbb{Z}$ such that
$$k \leq \tfrac{w_i}{t} +\tfrac{h_i s}{q-1} < k+1.$$
Then $0 \leq x:= \tfrac{w_i}{t} +\tfrac{h_i s}{q-1} -k <1$ and $(q-1)x \in \mathbb Z$.
So, by Theorem \ref{thm_GrossKoblitzMult}, with $h=h_i$ and $p \nmid h_i$, 
\begin{multline*}
\prod_{a=0}^{r-1} \prod_{b=0}^{h_i-1} \biggfp{\langle (\tfrac{w_i}{th_i} + \tfrac{s}{q-1} + \tfrac{b-k}{h_i}) p^a \rangle} \\
=\omega \bigl(h_i^{w_i \frac{q-1}{t} + h_i s}\bigr)
\prod_{a=0}^{r-1} \biggfp{\langle (\tfrac{w_i}{t} + \tfrac{h_i s}{q-1}) p^a \rangle} 
\prod_{b=1}^{h_i-1} \biggfp{\langle (\tfrac{b}{h_i}) p^a \rangle}.
\end{multline*}
As $\{b \mid b=0,1,\dots, h_i-1\} \equiv \{b-k \mid b=0,1,\dots, h_i-1\} \pmod {h_i}$ we have
\begin{multline}\label{for_gammaMult1}
\prod_{a=0}^{r-1} \prod_{b=0}^{h_i-1} \biggfp{\langle (\tfrac{w_i}{th_i} + \tfrac{s}{q-1} + \tfrac{b}{h_i}) p^a \rangle} \\
=\omega \bigl(h_i^{w_i \frac{q-1}{t} + h_i s}\bigr)
\prod_{a=0}^{r-1} \biggfp{\langle (\tfrac{w_i}{t} + \tfrac{h_i s}{q-1}) p^a \rangle} 
\prod_{b=1}^{h_i-1} \biggfp{\langle (\tfrac{b}{h_i}) p^a \rangle}.
\end{multline}
Similarly, with $k \in \mathbb{Z}$ chosen such that $0 \leq x:= \frac{w_i}{t} -k < 1$, we apply Theorem \ref{thm_GrossKoblitzMult} to get that
\begin{equation}\label{for_gammaMult2}
\prod_{a=0}^{r-1} \prod_{b=0}^{h_i-1} \biggfp{\langle (\tfrac{w_i}{th_i} + \tfrac{b}{h_i}) p^a \rangle} \\
=\omega \bigl(h_i^{w_i \frac{q-1}{t}}\bigr)
\prod_{a=0}^{r-1} \biggfp{\langle (\tfrac{w_i}{t}) p^a \rangle} 
\prod_{b=1}^{h_i-1} \biggfp{\langle (\tfrac{b}{h_i}) p^a \rangle}.
\end{equation}
Combining (\ref{for_gammaMult1}) and (\ref{for_gammaMult2}) we have, for $p \nmid h_i$,
\begin{equation}\label{for_gammaMult3}
\prod_{a=0}^{r-1} \biggfp{\langle (\tfrac{w_i}{t} + \tfrac{h_i s}{q-1}) p^a \rangle} 
=
\prod_{a=0}^{r-1} \prod_{b=0}^{h_i-1} \frac{\biggfp{\langle (\tfrac{w_i}{th_i} + \tfrac{s}{q-1} + \tfrac{b}{h_i}) p^a \rangle} }{\biggfp{\langle (\tfrac{w_i}{th_i} + \tfrac{b}{h_i}) p^a \rangle}}
\prod_{a=0}^{r-1} \biggfp{\langle (\tfrac{w_i}{t}) p^a \rangle}
\, \bar{\omega}^{s} \bigl(h_i^{h_i}\bigr).
\end{equation}
A final application of Theorem \ref{thm_GrossKoblitzMult}, this time with $k \in \mathbb{Z}$ such that $0 \leq x:= k - \frac{ds}{q-1} < 1$ and $p \nmid d$, we get, after some simplification, that
\begin{equation}\label{for_gammaMult4}
\prod_{a=0}^{r-1} \biggfp{\langle (\tfrac{-ds}{q-1}) p^a \rangle} 
= \prod_{a=0}^{r-1} \prod_{b=0}^{d-1} \frac{ \biggfp{\langle (\tfrac{-b}{d} - \tfrac{s}{q-1}) p^a \rangle}}{\biggfp{\langle (\tfrac{-b}{d}) p^a \rangle}}
\, \bar{\omega}^{s} \bigl(d^{-d}\bigr).
\end{equation}
Accounting for (\ref{for_gammaMult3}) and (\ref{for_gammaMult4}) in (\ref{for_Rw1})  and making the change of variable $s \to (q-1)-s$ we get that
\begin{multline}\label{for_Rw2}
R_{[w]} = (-1)^{n+1} \prod_{i=1}^{n}  \prod_{a=0}^{r-1}  \biggfp{\langle (\tfrac{w_i}{t}) p^a \rangle} 
\sum_{s=0}^{q-2} (-1)^{sd} 
\left[ \prod_{i=1}^{n}  \prod_{b=0}^{h_i-1} \prod_{a=0}^{r-1}  \frac{\biggfp{\langle (\tfrac{w_i}{th_i} + \tfrac{b}{h_i} - \tfrac{s}{q-1} ) p^a \rangle} }{\biggfp{\langle (\tfrac{w_i}{th_i} + \tfrac{b}{h_i}) p^a \rangle}} \right] \\
\times \left[  \prod_{b=0}^{d-1}  \prod_{a=0}^{r-1} \frac{ \biggfp{\langle (\tfrac{-b}{d} + \tfrac{s}{q-1}) p^a \rangle}}{\biggfp{\langle (\tfrac{-b}{d}) p^a \rangle}} \right]
 (-p)^{y} \, \bar{\omega}^{s} \bigl(\bigl[\lambda^d \prod_{i=1}^{n} h_i^{h_i} \bigr]^{-1}\bigr)
\end{multline}
where
\begin{equation*}
y =  \sum_{a=0}^{r-1}  \sum_{i=1}^{n}  (\tfrac{w_i}{t} ) p^a 
-\sum_{a=0}^{r-1}  \sum_{i=1}^{n} \left\lfloor  (\tfrac{w_i}{t} - \tfrac{h_i s}{q-1}) p^a \right\rfloor -\sum_{a=0}^{r-1} \left\lfloor  (\tfrac{ds}{q-1}) p^a \right\rfloor,
\end{equation*}
and we have used the fact that $\bar{\omega}(-1)=-1$.
Let
\begin{equation*}
z:= -\left[  \sum_{i=1}^{n}  \sum_{a=0}^{r-1} \sum_{b=0}^{h_i-1} \left\lfloor \langle (\tfrac{w_i}{th_i} + \tfrac{b}{h_i} ) p^a \rangle - \tfrac{sp^a}{q-1}   \right\rfloor
+ \sum_{a=0}^{r-1} \sum_{b=0}^{d-1} \left\lfloor \langle (\tfrac{-b}{d}) p^a \rangle + \tfrac{sp^a}{q-1}  \right\rfloor
 \right].
\end{equation*}
Using the fact that
$\left\lfloor mx \right\rfloor = \sum_{b=0}^{m-1} \left\lfloor x+ \frac{b}{m}  \right\rfloor$
we get that
\begin{multline*}
y-z =  \sum_{a=0}^{r-1}  \sum_{i=1}^{n}  (\tfrac{w_i}{t} ) p^a 
-\sum_{a=0}^{r-1}  \sum_{i=1}^{n}  \sum_{b=0}^{h_i-1} \left\lfloor  (\tfrac{w_i}{th_i} - \tfrac{s}{q-1}) p^a +\tfrac{b}{h_i} \right\rfloor -\sum_{a=0}^{r-1} \sum_{b=0}^{d-1} \left\lfloor \tfrac{sp^a}{q-1} + \tfrac{b}{d}\right\rfloor\\
+  \sum_{i=1}^{n}  \sum_{a=0}^{r-1} \sum_{b=0}^{h_i-1} \left\lfloor \langle (\tfrac{w_i}{th_i} + \tfrac{b}{h_i} ) p^a \rangle - \tfrac{sp^a}{q-1} \right\rfloor
+ \sum_{a=0}^{r-1} \sum_{b=0}^{d-1} \left\lfloor \langle (\tfrac{-b}{d}) p^a \rangle + \tfrac{sp^a}{q-1}  \right\rfloor
\end{multline*}
As $\gcd(p,d)=1$,  $\{b \mid b=0,1,\dots, d-1\} \equiv \{b p^a \mid b=0,1,\dots, d-1\} \pmod {d}$ and so
\begin{equation*}
\sum_{b=0}^{d-1} \left\lfloor \langle (\tfrac{-b}{d}) p^a \rangle + \tfrac{sp^a}{q-1}  \right\rfloor
=
\sum_{b=0}^{d-1} \left\lfloor \langle (\tfrac{b}{d}) p^a \rangle + \tfrac{sp^a}{q-1}  \right\rfloor
=
\sum_{b=0}^{d-1} \left\lfloor \langle \tfrac{b}{d} \rangle + \tfrac{sp^a}{q-1}  \right\rfloor 
=
\sum_{b=0}^{d-1} \left\lfloor \tfrac{b}{d} + \tfrac{sp^a}{q-1}  \right\rfloor.
\end{equation*}
Similarly, as $\gcd(p,h_i)=1$,
\begin{align*}
\sum_{b=0}^{h_i-1} \left\lfloor \langle (\tfrac{w_i}{th_i} + \tfrac{b}{h_i} ) p^a \rangle - \tfrac{sp^a}{q-1} \right\rfloor
&=
\sum_{b=0}^{h_i-1} \left\lfloor \langle (\tfrac{w_i}{th_i} ) p^a + \tfrac{b}{h_i} \rangle - \tfrac{sp^a}{q-1} \right\rfloor\\
&=
\sum_{b=0}^{h_i-1} \left\lfloor (\tfrac{w_i}{th_i} ) p^a + \tfrac{b}{h_i} - \left\lfloor(\tfrac{w_i}{th_i}) p^a + \tfrac{b}{h_i} \right\rfloor - \tfrac{sp^a}{q-1} \right\rfloor\\
&=
\sum_{b=0}^{h_i-1} \left\lfloor (\tfrac{w_i}{th_i} - \tfrac{s}{q-1}) p^a + \tfrac{b}{h_i}  \right\rfloor
-\sum_{b=0}^{h_i-1}  \left\lfloor(\tfrac{w_i}{th_i}) p^a + \tfrac{b}{h_i} \right\rfloor\\
&=
\sum_{b=0}^{h_i-1} \left\lfloor (\tfrac{w_i}{th_i} - \tfrac{s}{q-1}) p^a + \tfrac{b}{h_i}  \right\rfloor
-  \left\lfloor(\tfrac{w_i}{t}) p^a  \right\rfloor
\end{align*}
So
\begin{equation*}
y-z=  \sum_{a=0}^{r-1}  \sum_{i=1}^{n}  (\tfrac{w_i}{t} ) p^a  -  \sum_{a=0}^{r-1}  \sum_{i=1}^{n}  \left\lfloor(\tfrac{w_i}{t}) p^a  \right\rfloor
=  \sum_{a=0}^{r-1}  \sum_{i=1}^{n} \langle (\tfrac{w_i}{t} ) p^a \rangle.
\end{equation*}
Thus
\begin{align}\label{for_Rw3}
\frac{1}{q-1} R_{[w]}&=
\notag
\begin{multlined}[t][12cm]
(-1)^{n}
C(w)
\times \frac{-1}{q-1} \sum_{s=0}^{q-2} (-1)^{sd} 
\left[ \prod_{i=1}^{n}  \prod_{b=0}^{h_i-1} \prod_{a=0}^{r-1}  \frac{\biggfp{\langle (\tfrac{w_i}{th_i} + \tfrac{b}{h_i} - \tfrac{s}{q-1} ) p^a \rangle} }{\biggfp{\langle (\tfrac{w_i}{th_i} + \tfrac{b}{h_i}) p^a \rangle}} \right] \\
\times \left[  \prod_{b=0}^{d-1}  \prod_{a=0}^{r-1} \frac{ \biggfp{\langle (\tfrac{-b}{d} + \tfrac{s}{q-1}) p^a \rangle}}{\biggfp{\langle (\tfrac{-b}{d}) p^a \rangle}} \right]
 (-p)^{z} \, \bar{\omega}^{s} \bigl(\bigl[\lambda^d \prod_{i=1}^{n} h_i^{h_i} \bigr]^{-1}\bigr)
 \end{multlined}\\
 &= 
 \notag
 \begin{multlined}[t][11cm]
(-1)^{n} 
C(w) \;
{_{d}G_{d}}\biggl[ \begin{array}{ccccc} \multicolumn{5}{c}{\dotsc \dotsc \frac{w_i}{th_i} + \frac{b}{h_i} \dotsc \dotsc} \\[4pt]
 1 & \frac{1}{d} & \frac{2}{d} & \dotsc & \frac{d-1}{d} \end{array}
\Big| \; \bigl(\lambda^d  h_1^{h_1}  \dots h_n^{h_n} \bigr)^{-1} \; \biggr]_q.
 \end{multlined}\\[-23pt]
 \end{align}
\vspace{0pt}

\noindent Substituting for (\ref{for_Summand1}) and (\ref{for_Rw3}) in (\ref{for_NqDdlh}), we get 
 the required result.
\end{proof}

\begin{proof}[Proof of Corollary \ref{cor_Main1}]
In Theorem \ref{thm_Main}, we make the change of variables $w \to -w \pmod t$, which is a bijection on $W/\sim$, and $s \to (q-1)-s$ in the expansion of ${_{d}G_{d}}$ by definition.
 \end{proof}

\begin{proof}[Proof of Corollary \ref{cor_Main2}]
If $t=gcd(d,q-1)=1$ then $w=(0,0, \dots, 0)$ is the only element in $W$ and $C(0)=1$. So, by Corollary \ref{cor_Main1}
\begin{multline*}
N_q(D_{d, \lambda, h}) 
= \frac{q^{n-1}-1}{q-1}
+\frac{(-1)^n}{q} \left(-1+
{_{d}G_{d}}\biggl[ \begin{array}{ccccc} 0 & \frac{1}{d} & \frac{2}{d} & \dotsc & \frac{d-1}{d} \\[4pt]
 \multicolumn{5}{c}{\dotsc \dotsc \frac{b_i}{h_i} \dotsc \dotsc} \end{array}
\Big| \; \lambda^d  h_1^{h_1}  \dots h_n^{h_n}  \; \biggr]_q \right)
\end{multline*}
The first bottom line parameter in ${_{d}G_{d}}$ is $\frac{0}{h_1}=0$. We will ``cancel'' the zero from both top and bottom to get the required ${_{d-1}G_{d-1}}$. 
From Definition \ref{def_Gq} we see that the contribution to the summand of the top and bottom line zero is
\begin{equation*}
\prod_{k=0}^{r-1} 
\frac{\biggfp{\langle (0-\frac{s}{q-1} )p^k \rangle}}{\biggfp{\langle 0 p^k \rangle}}
\frac{\biggfp{\langle (0+\frac{s}{q-1}) p^k \rangle}}{\biggfp{\langle 0 p^k\rangle}}
(-p)^{-\lfloor{\langle 0p^k \rangle -\frac{s p^k}{q-1}}\rfloor -\lfloor{\langle 0 p^k\rangle +\frac{s p^k}{q-1}}\rfloor}
\end{equation*}
which, by Theorem \ref{thm_GrossKoblitz} and (\ref{for_GaussConj}), equals
\begin{equation*}
g(\bar{\omega}^{-s}) g(\bar{\omega}^s)
=
\begin{cases}
\bar{\omega}^s(-1) \, q & \textup{if } s \neq 0,\\
1 &  \textup{if }  s= 0.
\end{cases}
\end{equation*}
We also note that when $s=0$ the summand in Definition \ref{def_Gq} equals 1.
Therefore,
\begin{equation*}
{_{d}G_{d}}
\biggl[ \begin{array}{cccc} 0, & a_2, & \dotsc, & a_n \\
 0, & b_2, & \dotsc, & b_n \end{array}
\Big| \; \lambda \; \biggr]_q
=1 + q \cdot {_{d-1}G_{d-1}}
\biggl[ \begin{array}{ccc} a_2, & \dotsc, & a_n \\
 b_2, & \dotsc, & b_n \end{array}
\Big| \; \lambda \; \biggr]_q
\end{equation*}
as required.
\end{proof}

\begin{proof}[Proof of Theorem \ref{thm_Dwork}]
We start from Corollary \ref{cor_thm2_Dwork} and proceed in the same fashion as the second half of the proof of Theorem 2.2 in \cite{McC12}. We let $T = \bar{\omega}$ and apply the Gross-Koblitz formula, Theorem \ref{thm_GrossKoblitz}, to get
\begin{equation}\label{for_Dwork_NDl}
N_q(D_{\lambda}) 
= \frac{q^{n-1}-1}{q-1}
+\frac{1}{q(q-1)} \sum_{[w] \in W/\sim_1} R_{[w]}
\end{equation}
where
\begin{multline*}
R_{[w]} 
=\sum_{s=0}^{q-2} (-1)^{n+1} (-p)^{v} \,
\bar{\omega}^{ns}(-n \lambda) \,
 \prod_{a=0}^{r-1} \biggfp{\langle (\tfrac{-ns}{q-1} ) p^a \rangle}\\
\times
\prod_{k \in S_w^c} 
\bar{\omega}(-1)^{k\frac{q-1}{t}+s} \, q 
\prod_{a=0}^{r-1}
\frac{\biggfp{\langle (\tfrac{k}{t} + \tfrac{s}{q-1} ) p^a \rangle}^{n_k-1}}{\biggfp{\langle (\tfrac{-k}{t} - \tfrac{s}{q-1}) p^a \rangle}}
\end{multline*}
with
\begin{multline*}
v=\sum_{k \in S_w^c}  \frac{n_k k}{t} \sum_{a=0}^{r-1} p^a
- \sum_{k \in S_w^c} (n_k-1) \sum_{a=0}^{r-1} \left\lfloor (\tfrac{k}{t}+\tfrac{s}{q-1})p^a \right\rfloor \\
+  \sum_{k \in S_w^c}  \sum_{a=0}^{r-1}  \left\lfloor (-\tfrac{k}{t}-\tfrac{s}{q-1})p^a \right\rfloor
-   \sum_{a=0}^{r-1} \left\lfloor (\tfrac{-ns}{q-1})p^a \right\rfloor.
\end{multline*}
As $p \nmid n$ we derive from (\ref{for_gammaMult4}) that
\begin{equation*}
\prod_{a=0}^{r-1} \biggfp{\langle (\tfrac{-ns}{q-1}) p^a \rangle} 
= \prod_{a=0}^{r-1}  \frac{\displaystyle\prod_{k=0}^{t-1} \biggfp{\langle (\tfrac{k}{t} - \tfrac{s}{q-1}) p^a \rangle}    \displaystyle\prod_{\substack{b=0 \\ b \not\equiv 0 \imod{\frac{n}{t}}}}^{n-1} \biggfp{\langle (\tfrac{b}{n} - \tfrac{s}{q-1}) p^a \rangle}}{\displaystyle\prod_{b=0}^{n-1} \biggfp{\langle (\tfrac{b}{n}) p^a \rangle}}
\, \bar{\omega}^{s} \bigl(n^{-n}\bigr).
\end{equation*}
So, after some manipulation,
\begin{multline}\label{for_DworkR}
R_{[w]} 
= (-1)^{n+1} \sum_{s=0}^{q-2} (-p)^{v} \,
\bar{\omega}^{ns}(-\lambda) \,
\left[ \prod_{k \in S_w} \prod_{a=0}^{r-1} \frac{\biggfp{\langle (\tfrac{t-k}{t} - \tfrac{s}{q-1} ) p^a \rangle}}{\biggfp{\langle \tfrac{t-k}{t} p^a \rangle}} \right] \\
\times
\left[ \prod_{\substack{b=0 \\ b \not\equiv 0 \imod{\frac{n}{t}}}}^{n-1} \prod_{a=0}^{r-1} \frac{\biggfp{\langle (\tfrac{b}{n} - \tfrac{s}{q-1}) p^a \rangle}}{\biggfp{\langle \tfrac{b}{n} p^a \rangle}} \right] 
\left[ \prod_{k \in S_w^c} \prod_{a=0}^{r-1} \frac{\biggfp{\langle (-\tfrac{t-k}{t} + \tfrac{s}{q-1}) p^a \rangle}^{n_k-1}}{\biggfp{\langle -\tfrac{t-k}{t} p^a \rangle}^{n_k-1}} \right]
F[w]
\end{multline}
where
\begin{equation*}
F[w]:=
\left[ \prod_{k \in S_w^c} \prod_{a=0}^{r-1} \biggfp{\langle (\tfrac{-k}{t}) p^a \rangle}\right]^{-1}
\left[ \prod_{k \in S_w^c} \prod_{a=0}^{r-1} \biggfp{\langle \tfrac{k}{t} p^a \rangle}^{n_k-1} \right]
\left[ \prod_{k \in S_w^c} \bar{\omega}(-1)^{k\frac{q-1}{t}+s} \, q \right].
\end{equation*}
Applying the Gross-Koblitz formula, Theorem \ref{thm_GrossKoblitz}, in reverse and  (\ref{for_GaussConj}) we get that
\begin{align*}
 \prod_{k \in S_w^c} \prod_{a=0}^{r-1}  \biggfp{\langle (\tfrac{-k}{t}) p^a \rangle}  \biggfp{\langle \tfrac{k}{t} p^a \rangle}
&=  \prod_{k \in S_w^c} g(\bar{\omega}^{-k \frac{q-1}{t}}) g(\bar{\omega}^{k \frac{q-1}{t}}) 
(-p)^{-\sum_{a=0}^{r-1} \langle (\tfrac{-k}{t}) p^a \rangle +\langle (\tfrac{k}{t}) p^a \rangle}\\
&=   (-1)^{r |S_w^c\setminus\{0\}|} \prod_{k \in S_w^c} \bar{\omega}(-1)^{k\frac{q-1}{t}}.
\end{align*}
Thus
\begin{equation}\label{for_DworkF}
F[w]=
(-1)^{r |S_w^c\setminus\{0\}|} \, q^{|S_w^c|} \,  \bar{\omega}(-1)^{s |S_w^c|}
 \prod_{k \in S_w^c} \prod_{a=0}^{r-1} \biggfp{\langle \tfrac{k}{t} p^a \rangle}^{n_k}.
\end{equation}
If we let
\begin{multline*}
-z =
\sum_{k \in S_w^c} (n_k-1) \sum_{a=0}^{r-1} \left\lfloor \langle -\tfrac{t-k}{t} p^a \rangle + \tfrac{sp^a}{q-1} \right\rfloor \\
+\sum_{k \in S_w} \sum_{a=0}^{r-1} \left\lfloor \langle \tfrac{t-k}{t} p^a \rangle - \tfrac{sp^a}{q-1} \right\rfloor
+ \sum_{\substack{b=0 \\ b \not\equiv 0 \imod{\frac{n}{t}}}}^{n-1} \left\lfloor \langle \tfrac{b}{n} p^a \rangle - \tfrac{sp^a}{q-1} \right\rfloor,
\end{multline*}
then, after a lengthy but straightforward calculation, we find that
\begin{equation}\label{for_Dwork_vz}
v-z=-r \,|S_w^c\setminus\{0\}| + \sum_{i=1}^{n} \sum_{a=0}^{r-1} \langle \tfrac{w_i}{t} p^a \rangle.
\end{equation}
Accounting for (\ref{for_DworkF}) and (\ref{for_Dwork_vz}) in (\ref{for_DworkR}), and then (\ref{for_Dwork_NDl}), yields the result.
\end{proof}


\section{Concluding Remarks}\label{sec_cr}
When ${d \mid q-1}$ it is possible express the results of Koblitz, and those in this paper, in terms of hypergeometric functions over finite fields, as defined by Greene \cite{G}, or using a normalized version defined by the author \cite{McC6}. For example, see \cite{Go2, McC12, N} for related results. To extend these results beyond $q \equiv 1 \pmod d$ it is necessary to move to the $p$-adic setting as we have done in this paper. Other results where the $p$-adic hypergeometric function, ${{_m}G_m}$, is used to count points on certain hypersurfaces, which are special cases of the results in this paper, can be found in \cite{BRS, Go}.



\end{document}